\documentclass[12pt, reqno]{amsart}
\usepackage{amsmath, amsthm, amscd, amsfonts, amssymb, graphicx, color, mathrsfs}
\usepackage[bookmarksnumbered, colorlinks, plainpages]{hyperref}
\usepackage[all]{xy} 
\usepackage{slashed}

\newtheorem{theorem}{Theorem}[section]
\newtheorem{lemma}[theorem]{Lemma}

\newtheorem{corollary}[theorem]{Corollary}

\theoremstyle{definition}
\newtheorem{definition}[theorem]{Definition}
\newtheorem{example}[theorem]{Example}

\theoremstyle{remark}
\newtheorem{remark}[theorem]{Remark}
\numberwithin{equation}{section}

\begin{document}
\setcounter{page}{1}

\title[On the local Weyl formula on compact Lie groups]{A note on the local Weyl formula on compact Lie groups}

\author[D. Cardona]{Duv\'an Cardona}
\address{
  Duv\'an Cardona:
  \endgraf
  Department of Mathematics: Analysis, Logic and Discrete Mathematics
  \endgraf
  Ghent University, Belgium
  \endgraf
  {\it E-mail address} {\rm duvanc306@gmail.com, duvan.cardonasanchez@ugent.be}
  }
  
  \author[J. Delgado]{Julio Delgado}
\address{
  Julio Delgado:
  \endgraf
  Departamento de Matem\'aticas
  \endgraf
  Universidad del Valle
  \endgraf
  Cali-Colombia
    \endgraf
    {\it E-mail address} {\rm delgado.julio@correounivalle.edu.co}
  }

\author[M. Ruzhansky]{Michael Ruzhansky}
\address{
  Michael Ruzhansky:
  \endgraf
  Department of Mathematics: Analysis, Logic and Discrete Mathematics
  \endgraf
  Ghent University, Belgium
  \endgraf
 and
  \endgraf
  School of Mathematical Sciences
  \endgraf
  Queen Mary University of London
  \endgraf
  United Kingdom
  \endgraf
  {\it E-mail address} {\rm michael.ruzhansky@ugent.be, m.ruzhansky@qmul.ac.uk}
  }

\thanks{The authors were supported  by the FWO  Odysseus  1  grant  G.0H94.18N:  Analysis  and  Partial Differential Equations and  by the Methusalem programme of the Ghent University Special Research Fund (BOF)
(Grant number 01M01021). Duv\'an Cardona  was also supported  by the  DyCon Project 2015 H2020-694126.   Michael Ruzhansky is also supported  by EPSRC grant 
EP/R003025/2.
}


\begin{abstract}  In this note we reformulate  the spectral side of the Weyl law in the language of the matrix-valued quantisation on compact Lie groups.
\end{abstract} 

\maketitle

\allowdisplaybreaks

\section{Introduction}
Given a classical pseudo-differential operator $A$ of order zero on a compact Lie group $G,$ in this note we address the problem of reformulating  the spectral side of the Weyl law in the language of the matrix-valued symbols defined in \cite{Ruz}. Similar formulae are obtained for operators of positive order. Formulae of this type   are only known for the case of the  torus $G=\mathbb{T}^n,$ see e.g.   Zelditch \cite{Zelditch} and here we deal with their extension to the non-commutative setting. 

\subsection{Motivation: the quantum ergodicity problem} 
On a compact Riemannian manifold $(M,g)$, the local Weyl law of Zelditch establishes that  for a classical pseudo-differential operator $A\in \Psi^0(M),$ of order $0,$ in terms of the normalised spectral data $\{\phi_j,\lambda_j^2\}$ of the positive Laplacian $-\Delta_g$ on $(M,g)$ (that are, its system of eigenfunctions $\phi_j$ with corresponding eigenvalues $\lambda_j^2$)   the following   asymptotic expansion is valid (see Zelditch \cite{Zelditch}) 
\begin{equation}\label{Local:Weyl:Law}
    \sum_{\lambda_{j}\leq \lambda}(A\phi_j,\phi_j)=\left(\smallint_{T^*\mathbb{S}(M)}a_{0}(x,\xi')d\mu_L(x,\xi')\right)\lambda^n+O(\lambda^{n-1}),
\end{equation}for $\lambda$ big enough where $a_0\in C^\infty(T^*M)$ is the  principal symbol of the operator $A,$ and $d\mu_L(x,\xi)$ is the usual (un-normalised) Liouville measure on the spherical cotangent bundle $T^*\mathbb{S}(M)=\{(x,\xi)\in T^*M:\Vert\xi\Vert_g=1\}.$ 

When $M=G$ is a compact Lie group, in this note we address the problem of computing the left-hand side of \eqref{Local:Weyl:Law} in terms of the matrix-valued global symbol of $A$  as developed  in \cite{Ruz}.  Indeed, according to such a matrix-valued quantisation, to a pseudo-differential operator, one can  associate a unique and globally defined symbol on the phase space $G\times \widehat{G},$ with $\widehat{G}$ being the unitary dual of the group. Thus, the main problem in reformulating  the left-hand side of \eqref{Local:Weyl:Law} is to translate the spectral information $(A\phi_j,\phi_j)$ into an
expression involving matrix-valued symbols. We present this reformulation in Theorem  \ref{theorem}.

Before presenting  Theorem  \ref{theorem} and our other observations, we discuss our motivation, which arose  from the problem of understanding the local Weyl formula and its relation with the  quantum ergodicity problem.

 We have denoted the eigenvalues of the Laplacian by $\lambda_j^2.$ They are usually understood as different levels of energies $E_j=\lambda_j^2,$ and the most basic quantities testing the asymptotics of eigenfunctions are the matrix elements $(A\phi_j,\phi_j).$ These matrix elements measure the expected value of the observable $A$ (when the order of this operator is zero) in the energy state $\phi_j$.  Much of the work in quantum ergodicity is devoted to the study of the limits of $(A\phi_j,\phi_j),$ see e.g. Schnirelman \cite{Shnirelman}.  These limits are still the most accessible aspects of eigenfunctions. Indeed, accumulation points of the sequence $(A\phi_j,\phi_j)$ are called  quantum limits of the system $(\phi_j)$ with respect to the observable $A$.   A fundamental problem in quantum mechanics is to determine whether or not,  the sequence $(A\phi_j,\phi_j)$ has a unique quantum limit.  This is known as the quantum unique ergodicity problem.

The local Weyl formula is connected with the quantum unique ergodicity problem.  Indeed, if $N(\lambda)=\#\{j: \lambda_j\leq \lambda\}$ is the eigenvalue counting function of $\sqrt{-\Delta_g},$ then \eqref{Local:Weyl:Law} implies that
\begin{equation}\label{Quatum:limit}
  \lim_{\lambda\rightarrow\infty}  \frac{1}{N(\lambda)}\sum_{\lambda_{j}\leq \lambda}(A\phi_j,\phi_j)= \frac{1}{\textnormal{Vol}(T^*\mathbb{S}(M))}\smallint_{T^*\mathbb{S}(M)}a_{0}(x,\xi')d\mu_L(x,\xi),
\end{equation} where $\textnormal{Vol}(T^*\mathbb{S}(M))\neq 0$ denotes the volume of the co-sphere {\it bundle} when the closed manifold  $M\neq \emptyset$ is non-empty.   So, if the observable $A$ is a self-adjoint operator of order zero, there is a sub-sequence of matrix  elements $(A\phi_j,\phi_j)$  which converges to the quantum limit  $\textnormal{QL}:=\smallint_{T^*\mathbb{S}(M)}a_{0}(x,\xi')d\mu_L(x,\xi').$

\subsection{Main results}
In order to present our results we refer the reader to Section \ref{Section2:pre:l},  to  Subsection \ref{Notations}  and to Remark \ref{rem} for the notations used below. 

We denote by $\mathcal{L}_G$  the canonical positive Laplacian on $G.$ Its discrete spectrum  $\{\mu_\xi:[\xi]\in \widehat{G}\}$ can be enumerated with the unitary dual $\widehat{G}.$ Then, we use the notations $|\xi|:=\sqrt{{\mu}}_\xi$ and $N(\lambda)=\#\{[\xi]\in \widehat{G}:|\xi|\leq \lambda\}$ for the eigenvalues of $\sqrt{\mathcal{L}}_G$ as well as for its Weyl eigenvalue counting formula, respectively. 

Our main tool is the following formula
\begin{equation}\label{Main:obser}
      \sum_{i,j=1}^{d_\xi}d_\xi(A\xi_{ij},\xi_{ij})= d_\xi\smallint\limits_G\textnormal{Tr}(\sigma_{\textnormal{glob}}(x,\xi))dx
\end{equation} presented in Lemma \ref{main:lemma}, where $dx$ is the normalised Haar measure on $G,$ and  $A\in \Psi^0(G)$ is a  classical pseudo-differential operator of order zero, $\sigma_{\textnormal{glob}}(x,\xi)$ denotes its global matrix-valued symbol (c.f. \cite{Ruz}) the functions $\xi_{ij}$ are matrix-coefficients of a fixed unitary and irreducible representation $\xi$ of $G,$  and $d_\xi$ denotes its dimension.  As a consequence of \eqref{Main:obser}, we derive  the results below.

\begin{theorem}\label{theorem}
Let $G$ be a compact Lie group of dimension $n,$ and let $A\in \Psi^0(G)$ be a classical pseudo-differential operator of order zero. In terms of the following data:
\begin{itemize}
    \item first, with $\sigma_{\textnormal{loc}}$ being the (H\"ormander) principal symbol of $A,$
    \item denoting by $\sigma_{\textnormal{glob}}$  the matrix-valued symbol of $A$ defined on $G\times \widehat{G},$
    \item and with $\mu_{L}$ denoting the standard (un-normalised) Liouville measure on the co-sphere  bundle $T^*\mathbb{S}(G),$
\end{itemize}  for any $\lambda>0,$ the partial trace of $A$ admits the asymptotic expansion\small{
\begin{equation}\label{Symbol:formula¨localvsglobal}
    \sum_{[\xi]\in \widehat{G}:|\xi|\leq \lambda}d_\xi\smallint\limits_G\textnormal{Tr}(\sigma_{\textnormal{glob}}(x,\xi))dx=\left( \smallint\limits_{T^*\mathbb{S}(G)}\sigma_{\textnormal{loc}}(x,\xi')d\mu_L(x,\xi')\right)\lambda^n+O(\lambda^{n-1}).
\end{equation} }Moreover, the convergence at infinity of its average  with respect to the eigenvalue counting function $N(\lambda):=\#\{[\xi]\in \widehat{G}:|\xi|\leq \lambda\}$  is given by the quantum limit
\begin{equation}\label{Eq:Local:Weyl:Form}
    \lim_{\lambda\rightarrow\infty}\frac{1}{N(\lambda)}\sum_{ [\xi]\in \widehat{G}:|\xi|\leq \lambda }d_\xi\smallint\limits_G\textnormal{Tr}(\sigma_{\textnormal{glob}}(x,\xi))dx=\frac{1}{\textnormal{Vol}(T^*\mathbb{S}(G))}\smallint\limits_{T^*\mathbb{S}(G)}\sigma_{\textnormal{loc}}(x,\xi')d\mu_L(x,\xi').
\end{equation} 
\end{theorem}
Now, we present the following application of Theorem \ref{theorem} to quantum limits.
\begin{corollary}\label{corollary} Under the hypothesis of Theorem  \ref{theorem} and by assuming that $A$ is  self-adjoint (an observable)  on $L^2(G),$ there exists a sequence of unitary representations $\xi_j$ of the group $G$ such that
the sequence
$$\{d_{\xi_j}\smallint\limits_G\textnormal{Tr}(\sigma_{\textnormal{glob}}(x,\xi_j))dx\}_j$$ converges to the quantum limit $\textnormal{QL}:=\smallint\limits_{T^*\mathbb{S}(G)}\sigma_{\textnormal{loc}}(x,\xi')d\mu_L(x,\xi').$ 
\end{corollary}

In the following Theorem \ref{theorem:2}  we relax the usual ellipticity condition on the operator\footnote{See e.g. the discussion in Strichartz \cite{Strichartz}.} by imposing the positivity of its matrix-valued symbol (see Remark \ref{Remark}).
\begin{theorem}\label{theorem:2}
Let $A\in \Psi^m(G)$ be a classical pseudo-differential operator, let $m>0$ and let us assume that  its matrix-valued symbol is positive, i.e. 
\begin{equation}\label{Eq:ref:trace}
 \forall(x,[\xi])\in G\times \widehat{G},\,   \textnormal{Tr}(\sigma_{\textnormal{glob}}(x,[\xi]))\geq 0.   
\end{equation}
Then, for any $\lambda>0,$ the partial trace of $A$ satisfies the growth estimate\begin{equation}\label{Symbol:formula¨localvsglobal:2}
    \sum_{[\xi]\in \widehat{G}: |\xi|\leq \lambda}d_\xi\smallint\limits_G\textnormal{Tr}(\sigma_{\textnormal{glob}}(x,\xi))dx\leq C_{n,A}\lambda^{n+m}+O(\lambda^{n+m-1}),
\end{equation}where 
 $
 C_{n,A}= \textnormal{QL}:=  \smallint\limits_{T^*\mathbb{S}(G)}\sigma_{\textnormal{loc}}(x,\xi')d\mu_L(x,\xi').$
\end{theorem}

\begin{remark}
 In Theorem \ref{theorem:2} the matrix-valued symbol of the operator $A\in \Psi^{m}(G)$ satisfies the trace condition   \eqref{Eq:ref:trace} if e.g. on every representation space the matrix-valued symbol is a positive definite matrix i.e. if $\forall (x,[\xi])\in G\times \widehat{G},$ $\sigma_{\textnormal{glob}}(x,[\xi])\geq 0.$
\end{remark}

\section{Preliminaries}\label{Section2:pre:l}

\subsection{The Fourier analysis of a compact Lie group} Let $dx$ be the normalised Haar measure on a compact Lie group $G.$  
  The Hilbert space $L^2(G)$ will be endowed with
   the inner product $$ (f,g)=\smallint\limits_{G}f(x)\overline{g(x)}dx.$$   Next we present some basic facts about the representation theory of compact Lie groups.
\begin{definition}[Unitary representation of a compact Lie group]
    A continuous and unitary representation of  $G$ on $\mathbb{C}^{\ell}$ is any continuous mapping $\xi\in\textnormal{Hom}(G,\textnormal{U}(\ell)) ,$ where $\textnormal{U}(\ell)$ is the Lie group of unitary matrices of order $\ell\times \ell.$ The number $\ell=\dim{(\xi)}$ and also denoted by $\ell=d_\xi$ is called the dimension of the representation $\xi.$
\end{definition}
\begin{remark}[Irreducible representations] A subspace $W\subset \mathbb{C}^{d_\xi}$ is called $\xi$-invariant if for any $x\in G,$ $\xi(x)(W)\subset W,$ where $\xi(x)(W):=\{\xi(x)v:v\in W\}.$ The representation $\xi$ is irreducible if its only invariant subspaces are $W=\{0\}$ and $W=\mathbb{C}^{d_\xi}.$ Any unitary representation $\xi$ is a direct sum of unitary irreducible representations. 
\end{remark}
\begin{definition}[Equivalent representations]
    Two unitary representations $ \xi$  and $\eta$  are equivalent if there exists an invertible linear mapping $S:\mathbb{C}^{d_\xi}\rightarrow \mathbb{C}^{d_\eta}$ such that for any $x\in G,$ $S\xi(x)=\eta(x)S.$ The mapping $S$ is called an invertible  intertwining  operator between $\xi$ and $\eta.$  
\end{definition}

\begin{definition}[The unitary dual]
    The relation $\sim$ on the set of unitary representations $\textnormal{Rep}(G)$ defined by: {\it $\xi\sim \eta$ if and only if $\xi$ and $\eta$ are equivalent representations,} is an equivalence relation. The quotient 
$
    \widehat{G}:={\textnormal{Rep}(G)}/{\sim}
$ is called the unitary dual of $G.$
\end{definition}
The Fourier transform of a function is defined on unitary representations as follows.
\begin{definition}[Group Fourier transform]
    If $\xi\in \textnormal{Rep}(G),$ the Fourier transform $\mathscr{F}_{G}$ associates to any $f\in C^\infty(G)$ a matrix-valued function $\mathscr{F}_{G}f$ defined on $\textnormal{Rep}(G)$ as follows
$$ (\mathscr{F}_{G}f)(\xi) \equiv   \widehat{f}(\xi)=\smallint\limits_Gf(x)\xi(x)^{*}dx,\,\,\xi\in \textnormal{Rep}(G),\,\xi(x)^*:=\xi(x)^{-1}, \,x\in G. $$ Note that if $\mathbb{N}_0:=\mathbb{Z}\cap [0,\infty),$ then  $\mathscr{F}_{G}f:\textnormal{Rep}(G)\rightarrow \bigcup_{d\in \mathbb{N}_0}\mathbb{C}^{d\times d}$ is a matrix-valued mapping.
\end{definition}
\begin{remark}[The Fourier inversion formula on a compact Lie group]
The discrete Schwartz space $\mathscr{S}(\widehat{G}):=\mathscr{F}_{G}(C^\infty(G))$ is the image of the Fourier transform on the class of smooth functions. This operator admits a unitary extension from $L^2(G)$ into 
\begin{equation}
 \ell^2(\widehat{G})=\left\{\phi:\textnormal{Rep}(G)\rightarrow \bigcup_{d\in \mathbb{N}_0}\mathbb{C}^{d\times d}\,|\,\Vert \phi\Vert_{\ell^2(\widehat{G})}:=\left(\sum_{[\xi]\in \widehat{G}}d_{\xi}\Vert\phi(\xi)\Vert_{\textnormal{HS}}^2\right)^{\frac{1}{2}}<\infty \right\}.    
\end{equation} The norm $\Vert\phi(\xi)\Vert_{\textnormal{HS}}$ is the standard Hilbert-Schmidt norm of matrices.
The Fourier inversion formula takes the form
\begin{equation}\label{FIF}
  \forall x\in G,\,  f(x)=\sum_{[\xi]\in \widehat{G}} d_{\xi}\textnormal{Tr}[\xi(x)\widehat{f}(\xi)],\,f\in C^\infty(G),
\end{equation}where the summation is in such a way that from each  class $[\xi]$ we choose a unitary representation $\xi$. For proof of these facts see e.g. \cite[Proposition 10.3.17 and Theorem 10.3.23]{Ruz}.     
\end{remark}

\subsection{The quantisation formula} Next we present the aspects related to the matrix-valued quantisation as developed in  \cite{Ruz}. For this we require the following definition.
\begin{definition}[Right convolution kernel of an operator] Let  $A:C^\infty(G)\rightarrow C^\infty(G)$ be a continuous linear operator.
 The Schwartz kernel theorem associates to $A$ a kernel $K_A\in \mathscr{D}'(G\times G)$ such that
$    Af(x)=\smallint\limits_{G}K_{A}(x,y)f(y)dy,\,\,f\in C^\infty(G).$  The distribution defined via $R_{A}(x,xy^{-1}):=K_A(x,y)$ that provides the convolution identity
$$   Af(x)=\smallint\limits_{G}R_{A}(x,xy^{-1})f(y)dy,\,\,f\in C^\infty(G),$$
is called the right-convolution kernel of $A.$  
\end{definition}

\begin{remark} By following \cite[Theorem 10.4.4]{Ruz}, one can apply \eqref{FIF} to show that
  $A:C^\infty(G)\rightarrow C^\infty(G)$ can be written in terms of the Fourier transform as follows
 \begin{equation}\label{Quantisation:formula}
     Af(x)=\sum_{[\xi]\in \widehat{G}}d_\xi\textnormal{Tr}[\xi(x)\widehat{R}_{A}(x,\xi)\widehat{f}(\xi)],\,f\in C^\infty(G).
 \end{equation}Note that $\widehat{R}_{A}(x,\xi)$ denotes the Fourier transform of the distribution $y\mapsto \widehat{R}_{A}(x,y).$ In general, one needs to introduce a suitable space of distributions on  $G\times\widehat{G} $ with matrix-valued test functions to make sense of \eqref{Quantisation:formula}. 
 We introduce this space in the following definition. 
\end{remark}
\begin{definition}[Matrix-valued distributions defined on $G\times \widehat{G}$]\label{Def:sym} Let $$\sigma:G\times \widehat{G}\rightarrow \bigcup_{\ell\in \mathbb{N}_0}\mathbb{C}^{\ell\time \ell},$$ be a matrix-valued function such that for any $[\xi]\in \widehat{G},$ $\sigma(\cdot,[\xi])$ is of $C^\infty$-class, and such that, for any given $x\in G$ there is a distribution $R({x},\cdot)\in \mathscr{D}'(G),$ smooth in $x,$ satisfying $\sigma(x,\xi)=\widehat{R}({x},\xi),$ $[\xi]\in \widehat{G}$. Then we will say that $\sigma$ is a matrix-valued distribution on $G\times \widehat{G}.$
\end{definition} That any continuous linear operator $A:C^\infty(G)\rightarrow C^\infty(G)$ admits a unique {\it matrix-valued distribution} $\sigma_A,$ in the sense of Definition \ref{Def:sym}, that is, in such a way that for any $[\xi]\in \widehat{G},$ $\sigma_A(x,\xi)=\widehat{R}_A({x},\xi),$ can be found in  detail in \cite[Theorem 10.4.4]{Ruz}.
The following theorem summarises this fact by combining Theorems  10.4.4 and 10.4.6 of \cite[Pages 552-553]{Ruz}.
\begin{theorem}\label{The:quantisation:thm}
    Let $A:C^\infty(G)\rightarrow C^\infty(G) $ be a continuous linear operator. The following statements are equivalent.
    \begin{itemize}
        \item There exists a distribution $\sigma_A(x,[\xi]):G\times \widehat{G}\rightarrow \cup_{\ell\in \mathbb{N}_0}\mathbb{C}^{\ell\times \ell}$ that satisfies the quantisation formula
        \begin{equation}\label{Quantisation:3}
            \forall f\in C^\infty(G),\,\forall x\in G,\,\, Af(x)=\sum_{[\xi]\in \widehat{G}}d_\xi\textnormal{Tr}[\xi(x)\sigma_{A}(x,[\xi])\widehat{f}(\xi)].
        \end{equation}
        \item $ 
            \forall (x,[\xi])\in G\times \widehat{G},\, \sigma_{A}(x,\xi)=\widehat{R}_A(x,\xi).
        $ \\
        \item $ 
          \forall (x,[\xi]),\, \sigma_A(x,\xi)=\xi(x)^{*}A\xi(x),$ where $ A\xi(x):=(A\xi_{ij}(x))_{i,j=1}^{d_\xi}.  
        $ 
    \end{itemize}
\end{theorem}
Using the terminology in \cite{Ruz},  the function $\sigma_A$ in  Theorem \ref{The:quantisation:thm} is called the matrix-valued global symbol of $A.$ We will change this notation in the case where $A$ is a pseudo-differential operator on $G,$ using $\sigma_{\textnormal{glob}}$ for the global matrix-valued symbol $\sigma_A$ of $A$ in order to distinguish it from the standard principal symbol of $A,$ which we always denote by $\sigma_{\textnormal{loc}},$ see \eqref{Global:symbol} in Subsection \ref{Notations:2}. 

\begin{example}[Symbol of a Borel function of the Laplacian]\label{Example}  In this note, we consider the inner product $g(X,Y)=(X,Y)_{g}:=-\textnormal{Tr}[\textnormal{ad}(X)\textnormal{ad}(Y)]$ on $\mathfrak{g}.$ Note that $(\cdot,\cdot)_{g}$ is the negative of the Killing form.  Let $\mathbb{X}=\{X_1,\cdots,X_n\}$ be an orthonormal basis of $\mathfrak{g}$ with respect to $(\cdot,\cdot)_{g}$. The canonical positive Laplacian on $G$ is defined via
$
    \mathcal{L}_G=-\sum_{j=1}^nX_j^2,
$ 
and then is independent of the choice of  $\mathbb{X}.$  The spectrum of $\mathcal{L}_G$ is a discrete set that can be enumerated in terms of the set $\widehat{G}$ as
\begin{equation}
    \textnormal{Spect}(\mathcal{L}_G)=\{\mu_{[\xi]}:[\xi]\in \widehat{G}\}.
\end{equation}For a Borel function $f:\mathbb{R}^+_0\rightarrow \mathbb{C},$ consider the operator  $f(\mathcal{L}_G)$ defined by the spectral calculus. One can apply  Theorem \ref{The:quantisation:thm} in order to compute the symbol of $f(\mathcal{L}_G)$  as follows
\begin{equation}
    \sigma_{f(\mathcal{L}_G)}(x,\xi)=\xi(x)^*f(\mathcal{L}_G)\xi(x)=f(\mu_{[\xi]})I_{d_\xi}.
\end{equation}Note that the symbol  $\sigma_{f(\mathcal{L}_G)}(\xi)=\sigma_{f(\mathcal{L}_G)}(x,\xi)$ does not depend of the spatial variable $x\in G.$     
\end{example}
\subsection{Notations}\label{Notations:2}The following notation taken from \cite{Hormander1985III,Ruz,Zelditch} will be used in the next section. \label{Notations} 
 \begin{itemize}
\item For all $m\in \mathbb{R},$  $\Psi^m(G):=\Psi^m_{1,0}(G)$ denotes the Kohn-Nirenberg class of pseudo-differential operators of order $m,$ see H\"ormander \cite{Hormander1985III}.
\item   We denote by $|\xi|:=\sqrt{\mu}_{\xi},$ $[\xi]\in \widehat{G},$ the system of eigenvalues of $\sqrt{\mathcal{L}}_G.$

\item  The unit (co-) ball bundle is denoted by $T^*\mathbb{B}(G) = \{(x, \xi):
\|\xi\|_g \leq 1\}$. Its boundary $T^*\mathbb{S}(G)= \{\|\xi\|_g = 1\}$ is  the spherical vector bundle.

\item $d\mu_L$ is the {\it standard Liouville measure} on $T^*\mathbb{S}(G)$, i.e. the
surface measure $d\mu = \frac{dx d\xi}{d H}$  induced by the
Hamiltonian   $H = \|\xi\|_g$ and by the symplectic volume measure
$dx d\xi$ on $T^*G$.
\item We denote by $\sigma_{\textnormal{glob}}:G\times \widehat{G}\rightarrow \bigcup_{\ell\times \ell \in \widehat{G}}\textnormal{End}(\mathbb{C}^\ell)$  the matrix-valued  symbol of  a classical pseudo-differential operator $A\in \Psi^0(G).$ Then, by Theorem \ref{The:quantisation:thm}
\begin{equation}\label{Global:symbol}
 \forall [\xi]\in \widehat{G},\,\forall x\in G,\, \sigma_{\textnormal{glob}}(x,\xi)=\xi(x)^*A\xi(x).
\end{equation}
\end{itemize}

\section{The spectral side of the local Weyl formula on $G$}\label{Proof:mth} We start this section with the following remark.
\begin{remark}\label{rem}
    Let us choose  one (any) representative $\xi$ from each equivalence class $[\xi]\in \widehat{G}.$ 
Each $\xi:G\rightarrow\textnormal{End}(\mathbb{C}^{d_\xi})$ is a continuous mapping where $d_\xi=\dim(\xi)$ is its dimension. For any $x\in G,$ $\xi(x)$ is a unitary linear operator. In terms of a basis of $\mathbb{C}^{d_\xi}$, $\xi(x):\mathbb{C}^{d_\xi}\rightarrow \mathbb{C}^{d_\xi},$ can be identified with a unitary matrix
 $\xi(x)\cong (\xi_{ij}(x))_{i,j=1}^{d_\xi}.$

 If one considers the system of functions
\begin{equation}\label{Peter:Weyl:Basis}
   B:= \{d_\xi^{\frac{1}{2}}\xi_{ij}:1\leq i,j\leq d_\xi,\,[\xi]\in \widehat{G}\}
\end{equation}   Then $B$ consists of eigenfunctions of the positive Laplacian and it provides an orthonormal basis of $L^2(G),$ see e.g. \cite[Pages 488, 525]{Ruz}.
\end{remark}

\subsection{Local Weyl formula for zero order operators} The core observation of this note is the  algebraic formula \eqref{core:rr} of the following lemma.
\begin{lemma}\label{main:lemma}Under the assumptions of Theorem \ref{theorem} and with the notation in Remark \ref{rem} the following formula \begin{equation}\label{core:rr}
      \sum_{i,j=1}^{d_\xi}(A\xi_{ij},\xi_{ij})= \smallint\limits_G\textnormal{Tr}(\sigma_{\textnormal{glob}}(x,\xi))dx
\end{equation} holds.
\end{lemma}
\begin{proof}[Proof of Lemma \ref{main:lemma}] In view of  \eqref{Global:symbol} we have that
$ 
    A\xi(x)=\xi(x)\sigma_{\textnormal{glob}}(x,\xi),
$ for every $(x,[\xi])\in G\times  \widehat{G}.$ Note that for any $1\leq i,j\leq d_\xi,$
\begin{equation}\label{Eq:ref}
     A\xi_{ij}(x)=(\xi(x)\sigma_{\textnormal{glob}}(x,\xi))_{ij}=\sum_{k=1}^{d_\xi}\xi(x)_{ik}\sigma_{\textnormal{glob}}(x,\xi)_{kj}.
\end{equation}In view of the identity
\begin{equation}\label{Er;RT}
    (A(d_\xi^{\frac{1}{2}}\xi_{ij}),d_\xi^{\frac{1}{2}}\xi_{ij})=d_\xi(A\xi_{ij},\xi_{ij})=d_\xi\smallint\limits_{G}A\xi_{ij}(x)\overline{\xi(x)}_{ij}dx,
\end{equation} and plugging \eqref{Eq:ref} into \eqref{Er;RT} we obtain that
$$ 
   \Small{\sum_{i,j=1}^{d_\xi}d_\xi(A\xi_{ij},\xi_{ij})=   d_\xi  \sum_{i,j,k=1}^{d_\xi}\smallint\limits_{G}\xi(x)_{ik}\sigma_{\textnormal{glob}}(x,\xi)_{kj}\overline{\xi(x)}_{ji}^tdx=  d_\xi\smallint\limits_{G}\textnormal{Tr}[\xi(x)\sigma_{\textnormal{glob}}(x,\xi){\xi(x)}^*]dx.}
$$ 
Now, since the spectral trace is invariant under unitary transformations, we have that
$$   d_\xi\smallint\limits_{G}\textnormal{Tr}[\xi(x)\sigma_{\textnormal{glob}}(x,\xi){\xi(x)}^*]dx=d_\xi\smallint\limits_{G}\textnormal{Tr}[\sigma_{\textnormal{glob}}(x,\xi)]dx.$$ Since $d_\xi\geq 1,$ the proof of Lemma \ref{main:lemma} is complete.
\end{proof}

\begin{proof}[Proof of  Theorem \ref{theorem}] 
 In view of \eqref{Local:Weyl:Law} applied to the basis $(\phi_j)_j=B$  we have the asymptotic formula
\begin{equation}\label{Preliminary:formula}
  \sum_{|\xi|\leq \lambda}  \sum_{i,j=1}^{d_\xi}(A(d_\xi^{\frac{1}{2}}\xi_{ij}),d_\xi^{\frac{1}{2}}\xi_{ij})= \left( \smallint\limits_{T^*\mathbb{S}(G)}\sigma_{\textnormal{loc}}(x,\xi')d\mu_L(x,\xi')\right)\lambda^n+O(\lambda^{n-1}),
\end{equation}for any $\lambda>0,$ where $\sigma_{\textnormal{loc}}\in C^\infty(T^*G)$ is the principal symbol of $A,$ and  $\mu_{L}$ denotes the  Liouville measure on $T^*\mathbb{S}(G).$ Then, the validity of \eqref{Symbol:formula¨localvsglobal} follows from Lemma \ref{main:lemma}.
To prove the equality \begin{equation}
    \lim_{\lambda\rightarrow\infty}\frac{1}{N(\lambda)}\sum_{ [\xi]\in \widehat{G}:|\xi|\leq \lambda }d_\xi\smallint\limits_G\textnormal{Tr}(\sigma_{\textnormal{glob}}(x,\xi))dx=\frac{1}{\textnormal{Vol}(T^*\mathbb{S}(G))}\smallint\limits_{T^*\mathbb{S}(G)}\sigma_{\textnormal{loc}}(x,\xi')d\mu_L(x,\xi'),
\end{equation}where $N(\lambda):=\#\{[\xi]:|\xi|\leq \lambda\}$ is the spectral function of  $\sqrt{\mathcal{L}}_G$,  note that
\begin{equation}
   \frac{1}{N(\lambda)}\smallint\limits_Gd_\xi\textnormal{Tr}(\sigma_{\textnormal{glob}}(x,\xi))dx=\left( \smallint\limits_{T^*\mathbb{S}(G)}\sigma_{\textnormal{loc}}(x,\xi')d\mu_L(x,\xi')\right)\frac{\lambda^n}{N(\lambda)}+O(\lambda^{-1}).
\end{equation}Since  $$\lambda^n\times N(\lambda)^{-1}\rightarrow 1/(\textnormal{Vol}(T^*\mathbb{S}(G)))>0,\, \lambda\rightarrow \infty,$$ (for this fact see e.g. Strichartz \cite{Strichartz}) the term of the right-hand side can be neglected since it goes to zero when $\lambda\rightarrow\infty$ proving  \eqref{Eq:Local:Weyl:Form}.
\end{proof}
\subsection{Local Weyl formula for operators of positive order}  In this subsection $A\in \Psi^m(G)$ denotes a classical pseudo-differential operator of order $m>0$.
\begin{proof}[Proof of Theorem \ref{theorem:2}] Let us use the notation in Subsection \ref{Notations}. As $(\cdot,\cdot)_{g}$ induces an inner product on the dual $\mathfrak{g}^{*},$ denote by
  $ \|\xi\|_g = \sqrt{(\xi,\xi)_g}$  the length of a
co-vector.

$T^*G$ is parallelizable, that is $T^*G\cong G\times \mathfrak{g}^*.$ Then, the (H\"ormander) principal symbol of $ \mathcal{L}_{G}$ is given by (see e.g. Wallach \cite[Page 16]{Wallach1973})
\begin{equation}
\forall (x,\theta)\in (T^*G \setminus \{0\})\cong (G\times \mathfrak{g}^{*}\setminus \{0\}),\,\sigma_{\textnormal{loc}, \mathcal{L}_{G}}(x,\theta)=(\theta,\theta)_{g}=:\|\theta\|_{g}^2.
\end{equation}Moreover, using the version of  Strichartz \cite[Theorem 1]{Strichartz:2} of Seeley's functional calculus \cite{seeley}, we have that $  \sqrt{\mathcal{L}}_{G}\in \Psi^{1}(G) $  is a pseudo-differential operator of first order. 
Since $\sqrt{\mathcal{L}}_{G}$ is not invertible, let $\sqrt{\mathcal{L}}_{G}^{-1}$ be the inverse of $\sqrt{\mathcal{L}}_{G}$ on the orthogonal complement of its kernel, i.e. if $P_0$ is the $L^2$-orthogonal projection on $\textnormal{Ker}(\sqrt{\mathcal{L}}_{G}),$ then
$$ \sqrt{\mathcal{L}}_{G}^{-1}\sqrt{\mathcal{L}}_{G}=\sqrt{\mathcal{L}}_{G}\sqrt{\mathcal{L}}_{G}^{-1}=I-P_0,\,\,\forall f\in \textnormal{Ker}(\sqrt{\mathcal{L}}_{G}),\, \sqrt{\mathcal{L}}_{G}^{-1}f:=0.  $$
That this operator agrees with the operator $f(\mathcal{L}_G),$ defined by the spectral calculus where $f(t)=t^{-1},$ and that $\sqrt{\mathcal{L}}_{G}^{-1}\in \Psi^{-1}(G) $  can be found in Shubin \cite[Chapter II, Page 93]{Shubin}, see also the remark in  Strichartz \cite[Page 712]{Strichartz:2}. 

Let us consider the operator $\sqrt{\mathcal{L}}_{G}^{-m}:=(\sqrt{\mathcal{L}}_{G}^{-1})^m$ defined by the spectral calculus, and define the pseudo-differential operator
\begin{equation}
    \tilde{A}=A\sqrt{\mathcal{L}}_{G}^{-m}\in \Psi^{0}(G).
\end{equation} One also has the property (see Shubin \cite[Chapter II, Page 93]{Shubin}) 
$$ \sqrt{\mathcal{L}}_{G}^{-m}\sqrt{\mathcal{L}}_{G}^m=\sqrt{\mathcal{L}}_{G}^m\sqrt{\mathcal{L}}_{G}^{-m}=I-P_0,\,\,\forall f\in \textnormal{Ker}(\sqrt{\mathcal{L}}_{G}),\, \sqrt{\mathcal{L}}_{G}^{-m}f:=0.  $$ In view of the identity
$$A=AI=A(\sqrt{\mathcal{L}}_{G}^{-m}\sqrt{\mathcal{L}}_{G}^{m}+P_0)= \tilde{A}\sqrt{\mathcal{L}}_{G}^m+AP_{0}, $$
we have that for any $\xi\in \textnormal{Rep}(G),$  with $\xi\neq 1_{\widehat{G}},$ where $1_{\widehat{G}}$ denotes the irreducible trivial representation, one has that $P_0\xi(x)=0,$ and consequently
\begin{align*}
    \sigma_{\textnormal{glob}}(x,\xi)=\xi(x)^*(A\xi(x))=\xi(x)^*\tilde{A}\sqrt{\mathcal{L}}_{G}^m\xi(x)=\xi(x)^*(A\xi(x))=(\xi(x)^*\tilde{A}\xi(x))|\xi|^{m}I_{d_\xi}.
\end{align*}From this, note that the matrix-valued symbol $\tilde{\sigma}_{\textnormal{glob}}(x,\xi)=\xi(x)^*\tilde{A}\xi(x)$  of $\tilde{A},$ on every representation space (different from the trivial one, which corresponds to $|\xi|=0$) is given by
\begin{equation}
     \forall (x,\xi)\in G\times (\widehat{G}\setminus\{1_{\widehat{G}}\} ),\,\, \tilde \sigma_{\textnormal{glob}}(x,\xi)=\sigma_{\textnormal{glob}}(x,\xi)|\xi|^{-m}.
\end{equation}  In view of Theorem \ref{theorem} we have the local Weyl formula 
\begin{equation}\label{proof:2}
    \sum_{[\xi]\in \widehat{G}:|\xi|\leq \lambda}d_\xi\smallint\limits_G\textnormal{Tr}({\tilde\sigma}_{\textnormal{glob}}(x,\xi))dx=\left( \smallint\limits_{T^*\mathbb{S}(G)}{\tilde\sigma}_{\textnormal{loc}}(x,\xi')d\mu_L(x,\xi')\right)\lambda^n+O(\lambda^{n-1}),
\end{equation}where $\tilde \sigma_{\textnormal{glob}}$ and $\tilde \sigma_{\textnormal{loc}}$ denote  the matrix-valued symbol  and the (H\"ormander) local principal symbol of $\tilde A,$ respectively.
Note that on the co-sphere, $\|\theta\|_g=1$ for all $\theta\in \mathfrak{g}^*\setminus \{0\},$ and 
\begin{equation}
    \forall (x,\theta)\in (T^*G\setminus \{0\})\cong G\times \mathfrak{g}^*\setminus \{0\},\,\, \tilde \sigma_{\textnormal{loc}}(x,\theta)=\sigma_{\textnormal{loc}}(x,\theta)\|\theta\|_g^{-m}= \sigma_{\textnormal{glob}}(x,\theta).
\end{equation} Then, we have
\begin{align*}
   & \sum_{[\xi]\in \widehat{G}:0<|\xi|\leq \lambda}d_\xi |\xi|^{-m}\smallint\limits_G \textnormal{Tr}({\sigma}_{\textnormal{glob}}(x,\xi))dx=\left( \smallint\limits_{T^*\mathbb{S}(G)}{\tilde\sigma}_{\textnormal{loc}}(x,\xi')d\mu_L(x,\xi')\right)\lambda^n+O(\lambda^{n-1})\\
     &-\smallint_{G}\textnormal{Tr}[\tilde \sigma_{\textnormal{glob} }(x,1_{\widehat{G}})]dx\\
     &=\left( \smallint\limits_{\{(x,\xi'):\|\xi'\|_{g}=1\}}{\sigma}_{\textnormal{loc}}(x,\xi')\|\xi'\|_{g}^{-m}d\mu_L(x,\xi')\right)\lambda^n+O(\lambda^{n-1})-\smallint_{G}\textnormal{Tr}[\tilde \sigma_{\textnormal{glob} }(x,1_{\widehat{G}})]dx\\
     &=\left( \smallint\limits_{T^*\mathbb{S}(G)}{\sigma}_{\textnormal{loc}}(x,\xi')d\mu_L(x,\xi')\right)\lambda^n+O(\lambda^{n-1})-\smallint_{G}\textnormal{Tr}[\tilde \sigma_{\textnormal{glob} }(x,1_{\widehat{G}})]dx.
\end{align*}For all $0<|\xi|\leq \lambda,$ we have that $\lambda^{-m}\leq |\xi|^{-m}.$ In view of \eqref{Eq:ref:trace}  we have that
\begin{align*}
  \lambda^{-m}  \sum_{[\xi]\in \widehat{G}:0<|\xi|\leq \lambda}d_\xi \smallint\limits_G\textnormal{Tr}({\sigma}_{\textnormal{glob}}(x,\xi))dx&\leq\left( \smallint\limits_{T^*\mathbb{S}(G)}{\sigma}_{\textnormal{loc}}(x,\xi')d\mu_L(x,\xi')\right)\lambda^n+O(\lambda^{n-1})\\
   &-\smallint_{G}\textnormal{Tr}[\tilde \sigma_{\textnormal{glob} }(x,1_{\widehat{G}})]dx.
\end{align*}Consequently,
\begin{align*}
   \sum_{[\xi]\in \widehat{G}:0<|\xi|\leq \lambda}d_\xi \smallint\limits_G\textnormal{Tr}({\sigma}_{\textnormal{glob}}(x,\xi))dx&\leq \left( \smallint\limits_{T^*\mathbb{S}(G)}{\sigma}_{\textnormal{loc}}(x,\xi')d\mu_L(x,\xi')\right)\lambda^{n+m}+O(\lambda^{n+m-1})\\
   &-\smallint_{G}\textnormal{Tr}[\tilde \sigma_{\textnormal{glob} }(x,1_{\widehat{G}})]dx \lambda^m.
\end{align*} The analysis above  proves \eqref{Symbol:formula¨localvsglobal:2}. Indeed, note that $\smallint_{G}\textnormal{Tr}[\tilde \sigma_{\textnormal{glob} }(x,1_{\widehat{G}})]=0.$ To see this note that $1\in \textnormal{Ker}(\mathcal{L}_G),$ and the functional calculus implies that for any function $f\in \textnormal{Ker}(\mathcal{L}_G),$  $\sqrt{\mathcal{L}_G}^{-m}f\equiv 0. $   
Then, since $1_{\widehat{G}}(x)=1$ for any $x\in G,$ we have that
\begin{align*}
    \tilde \sigma_{\textnormal{glob} }(x,1_{\widehat{G}})=1_{\widehat{G}}(x)^{*}A\sqrt{\mathcal{L}_G}^{-m}1_{\widehat{G}}(x)=A\sqrt{\mathcal{L}_G}^{-m}1=0.
\end{align*}The proof of Theorem \ref{theorem:2} is complete.   
\end{proof}
\begin{remark}\label{Remark} Let $P_{\xi}(v_{1},\cdots,v_{\xi})=v_{1}$ be the projection operator in the first component on every representation space $\mathbb{C}^{d_\xi},$ $[\xi]\in \widehat{G}.$ Note that $P_{\xi}\geq 0,$ and define $A$ to be the operator associated to the symbol $\sigma_{\textnormal{glob}}(x,\xi)=|\xi|^m P_{\xi},$ $m>0.$ Note that $A$ satisfies the condition in \eqref{Eq:ref:trace} and then the hypothesis in Theorem  \ref{theorem:2}. However, $A$ is not elliptic, indeed its matrix-valued symbol is not invertible for any $|\xi|>0.$ This is an example of a non-elliptic operator that can be analysed with Theorem  \ref{theorem:2} but cannot be treated with the standard Weyl law.
\end{remark}

\begin{remark}
[Expected value of multiplication operators] 
Let $\varkappa\in C^\infty(G)$ be a smooth function. Let 
$
    A_{\varkappa}:L^2(G)\mapsto L^2(G),\, A_{\varkappa}f:=\varkappa f,
$ be the corresponding multiplication operator by $\varkappa.$  Let us analyse the behaviour of the expected values
\begin{equation}\label{Energy:states}
     (A(d_\xi^{\frac{1}{2}}\xi_{ij}),d_\xi^{\frac{1}{2}}\xi_{ij})=d_\xi(A\xi_{ij},\xi_{ij})=d_\xi({\varkappa}\xi_{ij},\xi_{ij})
\end{equation}
where any  $\phi_{ij,[\xi]}:=d_\xi^{\frac{1}{2}}\xi_{ij}$ is an element of  the basis $B$  in \eqref{Peter:Weyl:Basis}. According to Lemma \ref{main:lemma} we have the identity
\begin{equation}\label{mult}
    \frac{1}{N(\lambda)} \sum_{|\xi|\leq \lambda}  \sum_{i,j=1}^{d_\xi}d_\xi(A\xi_{ij},\xi_{ij})= \frac{1}{N(\lambda)}\sum_{|\xi|\leq \lambda}d_\xi\smallint\limits_G\textnormal{Tr}(\sigma_{\textnormal{glob}}(x,\xi))dx,
\end{equation}where $\sigma_{\textnormal{glob}}(x,\xi)$ is the global symbol of $A_\varkappa.$ Now, due to  \eqref{Global:symbol} we have that $\sigma_{\textnormal{glob}}(x,\xi)=\xi(x)^{*}\varkappa(x)\xi(x)=\varkappa(x) I_{d_\xi},\,[\xi]\in \widehat{G},\,x\in G,$
where $I_{d_\xi}$ is the identity operator in $\mathbb{C}^{d_\xi}.$ Consequently, 
\begin{equation}
  \lim_{\lambda\rightarrow\infty}  \frac{1}{N(\lambda)}\sum_{|\xi|\leq \lambda}d_\xi\smallint\limits_G\textnormal{Tr}(\sigma_{\textnormal{glob}}(x,\xi))dx= \lim_{\lambda\rightarrow\infty}\frac{1}{N(\lambda)} \sum_{|\xi|\leq \lambda} d_\xi^2 \smallint\limits_{G}\varkappa(x)dx=\smallint\limits_{G}\varkappa(x)dx.
\end{equation}In view of  Theorem \ref{theorem}  we deduce that following simplified version of \eqref{Eq:Local:Weyl:Form}
\begin{equation}
    \frac{1}{\textnormal{Vol}(T^*\mathbb{S}(M))} \smallint\limits_{T^*\mathbb{S}(G)}\varkappa(x)d\mu_L(x,\xi')=\smallint\limits_{G}\varkappa(x)dx.
\end{equation}Note that, if the function $\varkappa\in C^\infty(G)$ is real-valued, the convergence of the average \eqref{mult} when $\lambda\rightarrow\infty,$ implies the existence of a sub-sequence $d_\xi^{\frac{1}{2}}\xi_{i_k,j_k}$ of energy states in $B$ such that $\lim_{k\rightarrow\infty}d_\xi(\varkappa\xi_{i_k,j_k},\xi_{i_k,j_k})=\smallint\limits_{G}\varkappa(x)dx,$ which shows that the Haar integral $\smallint\limits_{G}\varkappa(x)dx$ is a quantum limit of the sequence $B$ with respect to $A_\varkappa$.
\end{remark}
\begin{proof}[Proof of Corollary \ref{corollary}] Let us use a numerable correspondence between $\mathbb{N}$ and the unitary dual $\widehat{G}=\{[\xi_j]:j\in \mathbb{N}\}.$ Let us define the sequence
$$A_n:=d_{\xi_n}\smallint\limits_{G}\textnormal{Tr}[\sigma_{\textnormal{op}}(x,\xi_n)]dx.  $$ Let $S_n$ be the average sums of the sequence $\{A_{n}\}_{n\in \mathbb{N}}.$ Note that from 
Theorem \ref{theorem} we have the identity
\begin{equation}
    \lim_{n\rightarrow\infty}\frac{1}{N(n)}\sum_{[\xi_j]:|\xi_j|\leq n } A_j= \lim_{n\rightarrow\infty}\frac{1}{N(n)}S_{N(n)} =\smallint\limits_{T^*\mathbb{S}(G)}\sigma_{\textnormal{loc}}(x,\xi')d\mu_L(x,\xi').
\end{equation}Then, there exists a sub-sequence $\{A_{n_j}\}$ such that when $j\rightarrow\infty,$ 
\begin{equation}
    A_{n_j}\rightarrow \smallint\limits_{T^*\mathbb{S}(G)}\sigma_{\textnormal{loc}}(x,\xi')d\mu_L(x,\xi').
\end{equation}The proof of Corollary \ref{corollary} is complete.    
\end{proof}
{\bf Acknowledgement.} The authors would like to thank the reviewer for his/her comments which helped to improve the final version of this note.

\bibliographystyle{amsplain}

\end{document}